\documentclass[a4paper]{amsart}

\usepackage[utf8]{inputenc}
\usepackage[T1]{fontenc}
\usepackage{lmodern, enumerate}
\usepackage{amssymb,amsxtra}
\usepackage[all]{xy}
\usepackage{nicefrac,mathtools}
\usepackage{microtype}
\usepackage{amscd}
\usepackage{xcolor}
\usepackage[pdftitle={...},
 pdfauthor={},
 pdfsubject={Mathematics}
]{hyperref}
\usepackage[lite]{amsrefs}

\newtheorem{theorem}{Theorem}[subsection]
\newtheorem{lemma}[theorem]{Lemma}
\newtheorem{corollary}[theorem]{Corollary}
\newtheorem{proposition}[theorem]{Proposition}

\theoremstyle{definition}

\newtheorem{remark}[theorem]{Remark}
\newtheorem{example}[theorem]{Example}
\newtheorem{question}[theorem]{Question}

\numberwithin{equation}{section}
\numberwithin{theorem}{section}

\DeclareMathOperator{\Aut}{Aut}

\DeclareMathOperator{\tr}{\mathrm{Tr}}
\DeclareMathOperator{\alg}{\mathrm{alg}}

\newcommand{\ep}{\varepsilon}

\newcommand*{\CC}{\mathbb C}

\newcommand*{\ZZ}{\mathbb Z}
\newcommand*{\RR}{\mathbb R}
\newcommand*{\NN}{\mathbb N}
\newcommand*{\QQ}{\mathbb Q}
\newcommand*{\TT}{\mathbb T}

\newcommand*{\id}{\textup{id}}
\newcommand*{\Ad}{\textup{Ad} \, }

\newcommand{\Cs}{$C^*$-al\-ge\-bra}

\newcommand*{\into}{\hookrightarrow}

\newcommand{\SL}{\operatorname{SL}}

\newcommand{\Hom}{\operatorname{Hom}}

\title{Inclusions of $C^*$-algebras arising from fixed-point algebras}

\author{Siegfried Echterhoff and Mikael R\o rdam}
\thanks{SE funded by Deutsche Forschungsgemeinschaft (DFG, German Research Foundation) Project-ID 427320536 SFB 1442 and under Germany's Excellence Strategy EXC 2044  390685587, Mathematics M\"{u}nster: Dynamics, Geometry, Structure. 
MR supported by a research grant from the Danish Council for Independent Research, Natural Sciences.}



\address[echters@uni-muenster.de]{Siegfried Echterhoff,
Mathematisches Institut, Westf\"{a}lische Wilhelm-Universit\"{a}t M\"{u}nster, Einsteinstr. 62, 48149 Münster, Germany}

\address[rordam@math.ku.dk]{Mikael Rørdam, Department of Mathematical Sciences, University of Copenhagen, 
Universitetsparken 5, 2100 Copenhagen, Denmark}

\begin{document}

\begin{abstract}
We examine inclusions of $C^*$-algebras of the form $A^H \subseteq A \rtimes_{r} G$, where $G$ and $H$ are groups acting on a unital simple \Cs{} $A$ by outer automorphisms and $H$ is finite. It follows from a theorem of Izumi that $A^H \subseteq A$ is $C^*$-irreducible, in the sense that all intermediate \Cs s are simple. We show that $A^H \subseteq  A \rtimes_{r} G$ is $C^*$-irreducible for all $G$ and $H$ as above if and only if $G$ and $H$ have trivial intersection in the outer automorphisms of $A$, and we give a Galois type classification of all intermediate \Cs s in the case when $H$ is abelian and the two actions of $G$ and $H$ on $A$ commute.

We illustrate these results with examples of outer group actions on the irrational rotation \Cs s. We exhibit, among other examples, $C^*$-irreducible inclusions of AF-algebras that have intermediate \Cs s that are not AF-algebras, in fact, the irrational rotation \Cs{} appears as an intermediate \Cs.
\end{abstract}
%
%

\maketitle

\section{Introduction}

\noindent Inclusions of unital simple \Cs s with the property that all intermediate \Cs s are simple were characterized and labelled $C^*$-irreducible in the recent paper \cite{Rordam} by the second named author. A well-known and classic result of Kishimoto, \cite{Kishimoto}, states that whenever a group $G$ acts by outer automorphisms on a simple \Cs{} $A$, then the reduced crossed product $A \rtimes_r G$ is simple as well. It follows easily from the proof of this theorem that the inclusion $A \subseteq A \rtimes_r G$ is $C^*$-irreducible, when $A$ in addition is unital, cf.\ \cite[Theorem 5.8]{Rordam}.   Moreover,  Izumi, \cite[Corollary 6.6]{Izumi},  in the case of finite $G$, and Cameron and Smith, \cite[Theorem 3.5]{Cameron-Smith}, in the general case established a Galois correspondence between intermediate \Cs s $A\subseteq D\subseteq A\rtimes_rG$ and subgroups $L$ of $G$, via $L \mapsto D =  A\rtimes_r L$.

 It was observed by Rosenberg, \cite{Rosenberg}, that if $H$ is any finite group acting (outer or not) on any \Cs{} $A$, then $A^H$ is isomorphic to a hereditary sub-\Cs{} of $A \rtimes H$. In particular, if $A$ is simple and the action of $H$ on $A$ is by outer automorphisms, then $A^H$ is simple. A result of Izumi, \cite[Corollary 6.6]{Izumi}, shows that the inclusion $A^H\subseteq A$ then is $C^*$-irreducible and that 
 all intermediate algebras are of the form $A^H\subseteq A^L\subseteq A$ for subgroups $L$ of $H$. This  mirrors the situation of crossed products by finite groups, and Izumi indeed directly relates the fixed-point algebra inclusion to the corresponding crossed-product inclusion via a version of Jones basic construction (see \cite[Corollary 3.12]{Izumi}).

Bisch and Haagerup considered in their paper \cite{BH1996} subfactors of the form $P^H \subseteq P \rtimes G$ arising from outer actions of two finite groups $H$ and $G$ on a II$_1$-factor $P$. They show that certain properties of the resulting subfactors (finite depth, respectively, amenability) are precisely mirrored by properties of the subgroup of $\mathrm{Out}(P)$ generated by $H$ and $G$.  They also show that the inclusion $P^H \subseteq P \rtimes G$ is irreducible if and only if $G$ and $H$ intersect trivially in $\mathrm{Out}(P)$.

%

Specifially, as stated in the abstract, we prove in this paper that if $\alpha$ and $\beta$ are actions of groups $G$ and $H$  on a unital simple \Cs{} $A$, and if $H$ is finite, then the inclusion $A^H \subseteq A \rtimes_r G$ is $C^*$-irreducible if and only if $\alpha_s \circ \beta_t$ is outer for all $(s,t) \in G \times H$ with $(s,t) \ne (e_G,e_H)$. This condition is an exact translation to the realm of \Cs s of the Bisch-Haagerup condition ensuring irreducibility in the subfactor case.  In the case where $H$ is abelian and the two actions $\alpha$ and $\beta$ commute, we further establish a Galois correspondence between intermediate \Cs s of the inclusion $A^H \subseteq A \rtimes_r G$  and subgroups of $\widehat{H}\times G$, where $\widehat{H}$ denotes the Pontryagin dual of $H$. Clearly, $A$ itself is an intermediate \Cs{} of this inclusion.

%
 
 We apply our  results to some well-known outer actions of finite and infinite cyclic groups on the irrational rotation \Cs{} $A_\theta$. There is a canonical (outer) action of the group $\SL(2,\ZZ)$ on $A_\theta$. It is known that $\ZZ_2$, 
$\ZZ_3$, $\ZZ_4$ and $\ZZ_6$ are finite cyclic subgroups of $\SL(2,\ZZ)$, and in fact the only ones, up to conjugacy. The corresponding actions of these finite cyclic groups on $A_\theta$ were studied in \cite{ELPW}, and it was shown  therein, that the fixed-point algebra and the crossed product of $A_\theta$ by each of these groups gives rise to a simple AF-algebra. We use this, and our main result stated above, to show that when $F_1$ and $F_2$ are (certain) combinations of the groups  $\ZZ_2$, 
$\ZZ_3$ and $\ZZ_4$, then $A_\theta^{F_1} \subseteq A_\theta \rtimes F_2$ is a $C^*$-irreducible inclusion of simple AF-algebras admitting a non-AF intermediate \Cs, namely $A_\theta$. This answers in the negative Question 6.11 from \cite{Rordam}. We also study several interesting examples of $C^*$-irreducible inclusions which involve  actions of the integer group $\ZZ$. 

The paper is organized as follows. In Section~\ref{sec:outer} we collect some well-known and some new results about outer actions of groups on \Cs s. In Section~\ref{sec:fixed} we prove our main result on $C^*$-irreducibility of  inclusions of the form $A^H \subseteq A \rtimes_r G$, and in Section~\ref{sec-Galois} we establish the Galois correspondence for the intermediate subalgebras of these inclusions (under the assumptions stated above). Finally,  in Section~\ref{sec:examples} we provide examples of our main results relating to actions on the irrational rotation \Cs s. 

{\bf Aknowledgement:} The authors would like to thank Masaki Izumi and Stefaan Vaes for very helpful comments which helped to substantially improve the paper. In particular we would like to thank Masaki Izumi for pointing out to us the result \cite[Corollary 6.6]{Izumi} and a modification of 
our Lemma \ref{lm:OP} which resulted in a major improvement of our main results.

\section{Outer actions on fixed-point algebras} \label{sec:outer}

\noindent In this section we derive some preliminary results on outer actions of a discrete group $G$  on a \Cs{} $A$. 
The \Cs{} $A$ may or may not be unital, and if it is not unital we shall consider its multiplier algebra $M(A)$. 
For a unital \Cs{} $A$ we let $U(A)$ denote its group of unitary elements.

We shall repeatedly use the classic result by Kishimoto from \cite[Theorem 3.1]{Kishimoto} mentioned in the introduction that if $\alpha \colon G \to \Aut(A)$ is an action  of a discrete group $G$ by  outer automorphisms on a simple $C^*$-algebra $A$, then the reduced crossed product $A\rtimes_{\alpha,r}G$  is simple as well.  We shall often write $A\rtimes_\alpha G$ instead of $A\rtimes_{\alpha,r}G$ if $G$ is known to be amenable (in particular, if $G$ is abelian  or finite), since then the full and reduced crossed products coincide.  Also, we may write $A\rtimes_rG$ instead of $A\rtimes_{\alpha,r}G$ if the action $\alpha$ is understood. Recall that if $G$ is discrete there is always a canonical inclusion $A\subseteq A\rtimes_{\alpha,r}G$ together with  a canonical unitary representation 
$u \colon G\to UM(A\rtimes_{\alpha,r}G)$ implementing the action $\alpha$, i.e., $\alpha_g= \Ad u_g$, for $g \in G$. The {\em algebraic crossed product}
$$A\rtimes_{\alpha,\alg}G:=\Big\{\sum_{g\in G} a_gu_g : a_g\in A, a_g=0\;\text{for all but finitely many $g$}\Big\}$$
becomes a dense subalgebra of $A\rtimes_{\alpha,r}G$, and the two algebras coincide  if $G$ is finite.

Recall that an action $\alpha$ is {\em outer} if no $\alpha_g$ is inner, for $g \ne e$, that is $\alpha_g\neq \Ad v$ for all unitaries $v\in M(A)$. On the other extreme, if the action $\alpha \colon G\to\Aut(A)$ is implemented 
by a unitary representation $v \colon G\to UM(A)$ such that $\alpha_g=\Ad v_g$, for all $g\in G$, we have 
$$A\rtimes_{\alpha,r} G\cong A\rtimes_{\id, r}G\cong A\otimes C_r^*(G)$$
where the first isomorphism is the extension of the map 
$$A\rtimes_{\alpha,\alg}G\to A\rtimes_{\id,\alg}G\colon a_gu_g\mapsto (a_gv_g)u_g.$$
We use these results to prove

\begin{lemma}\label{lem-outer}
Let $\alpha \colon G\to\Aut(A)$ be an action of a discrete group on a simple $C^*$-algebra $A$. 
 Then the following are equivalent:
\begin{enumerate}
\item The action $\alpha$ is outer.
\item For all subgroups $H$ of $G$, the crossed product $A\rtimes_{\alpha,r}H$ is simple.
\item For all (finite or infinite) cyclic subgroups $C_g:=\langle g\rangle$ of $G$, the crossed product $A\rtimes_{\alpha}C_g$ is simple.
\end{enumerate}
\end{lemma}

\begin{proof} The implication (i) $\Rightarrow$ (ii) is a direct consequence 
 Kishimoto's theorem, since outerness of $\alpha$ implies outerness of the restriction of $\alpha$ to any subgroup of $G$. 
 The implication (ii) $\Rightarrow$ (iii) is trivial. Thus it suffices to prove (iii) $\Rightarrow$ (i). 

So assume that (iii) holds for all $g\in G$. If $\alpha$ is not outer,
there exists an element $e\neq g\in G$ such that $\alpha_g(a)=\Ad u(a)=uau^*$ for some unitary element 
$u\in M(A)$. Let $C_g$ be the cyclic subgroup of $G$ generated by $g$. 
Suppose first that $g$ has infinite order.
Since $\alpha_{g^n}=\Ad u^n$ for all $n\in \ZZ$, it follows that the restriction of 
$\alpha$ to $C_g\cong \ZZ$ is implemented by the unitary representation $n\mapsto u^n\in UM(A)$, and hence we get
$$A\rtimes_\alpha C_g\cong A\otimes C^*(C_g)\cong A\otimes C^*(\ZZ)\cong A\otimes C(\TT),$$
which is certainly not simple.

On the other hand, if $C_g$ is cyclic of order $m\in \NN$, then $\Ad u^m=\alpha_e=\id_A$. It follows from simplicty of $A$ that $A' \cap M(A) = \CC$,  so there must exists
 $\omega \in \TT$ such that $u^m=\omega1$. Now, if $\eta \in \TT$ is an $m$th root of $\bar{\omega}$, we see that 
$g^k\mapsto (\eta u)^k\in UM(A)$ implements a homomorphism $\tilde u \colon C_g\to UM(A)$ such that $\alpha|_{C_g}=\Ad \tilde u$, and hence
$$A\rtimes_\alpha C_g\cong A\otimes C^*(C_g)\cong A\otimes \CC^m,$$
which is not simple.
\end{proof}

 \begin{remark}\label{rem-outer}
 In general, outerness for an action  $\alpha \colon G\to\Aut(A)$ on a simple $C^*$-algebra 
$A$  (unital or not) is not equivalent to $A\rtimes_{\alpha,r}G$ being simple, even if $G$ is finite and abelian and $A$ is simple and unital. 
To construct a counterexample, let $H$ be any finite abelian group.
Let $G:=H\times \widehat{H}$ be the direct  product of $H$ with its dual group $\widehat{H}$.
For each pair $(g,x)\in H\times \widehat{H}$ let $V_{(g,x)}$ be the unitary operator on $\ell^2(H)$ 
defined by
$$\big(V_{(g,x)}\xi\big)(h)=\overline{\langle h, x\rangle}\xi(g^{-1}h),$$
where $\langle \, \cdot \, , \, \cdot \, \rangle \colon H\times \widehat{H}\to\TT$ denotes the canonical pairing between $H$ and $\widehat{H}$.
A short computation then shows that $V \colon H\times \widehat{H}\to U(\ell^2(H))$ is a projective representation 
such that
$$V_{(g_1,x_1)}V_{(g_2,x_2)}=\langle g_1, x_2\rangle V_{(g_1g_2,x_1x_2)},$$
for all $(g_1,x_1),(g_2,x_2)\in H\times \widehat{H}$. Thus, $V$ is an $\omega$-representation of the 
Heisenberg-type $2$-cocycle $\omega \colon H\times \widehat{H}\to \TT$ defined by
$\omega\big((g_1,x_1), (g_2,x_2)\big) =\langle g_1, x_2\rangle$. Let $C^*(H\times\widehat{H},\omega)$ denote the twisted group algebra 
of $H\times \widehat{H}$ with respect to the cocycle $\omega$ (see, e.g.,  \cite[Section 2.8.6]{CELY} for the construction).
Since $\omega$ is totally skew in the sense of \cite[p.~300]{Baggett-Kleppner} it follows from 
\cite[Theorem 3.3]{Baggett-Kleppner} that $V$ is the unique irreducible $\omega$-representation of $H\times\widehat{H}$,
which then implements an isomorphism $C^*(H\times\widehat{H},\omega)\cong B(\ell^2(H))\cong M_{|H|}(\CC)$. 

Now let $A:=B(\ell^2(H))$ and define $\beta \colon H\times \widehat{H} \to \Aut(A)$ by
$\beta_{(g,x)}=\Ad V^*_{(g,x)}$.
Then  one  checks that $A\otimes C^*(H\times \widehat{H},\omega)$ is isomorphic to $A\rtimes_\beta(H\times\widehat{H})$
via the map $a \otimes \delta_{(g,x)}\mapsto aV_{(g,x)}u_{(g,x)}$ (see, e.g., \cite[Remark 2.8.18]{CELY}). 
Thus  $\beta$ is an action by inner automorphisms on the simple unital \Cs{} $A = M_{|H|}(\CC)$ for which
$A\rtimes_\beta (H\times\widehat{H})\cong M_{|H|}(\CC)\otimes M_{|H|}(\CC)$ is simple.
\end{remark}

\section{$C^*$-irreducible inclusions arising from fixed-point algebras into crossed products} \label{sec:fixed}

\noindent We shall here prove our main results regarding $C^*$-irreducibility of inclusions arising from fixed-point algebras into crossed products.  Let
 $H$ be a finite group and let $\beta \colon H\to \Aut(A)$ be an action of $H$ on the $C^*$-algebra $A$. 
Let $$A^{H,\beta}:=\{a\in A: \beta_h(a)=a\; \text{for all} \;  h\in H\}$$ (or simply $A^H$ if  confusion seems unlikely)
be the fixed-point algebra of $\beta$. Consider the projection 
\begin{equation} \label{eq:p-beta}
p^\beta:=\frac{1}{|H|}\sum_{h\in H}u_h\in M(A\rtimes_\beta H),
\end{equation}
where $u \colon H\to UM(A\rtimes_\beta H)$ denotes the canonical unitary representation  which implements $\beta$ in the crossed-product.
Note that $p^\beta$ commutes with $A^H$. Rosenberg observed  in \cite{Rosenberg} that the image of the $^*$-homomorphism ${A^H \ni a} \mapsto ap^\beta = \frac{1}{|H|}\sum_{h\in H}a u_h\in A\rtimes_\beta H$ is equal to $p^\beta(A \rtimes_\beta H)p^\beta$, so that we get an isomorphism
\begin{equation} \label{eq:Rosenberg}
A^H \cong p^\beta(A \rtimes_\beta H)p^\beta.
\end{equation}

We say that $\beta$ is {\em saturated} if $Ap^{\beta}A$ (or $p^{\beta}$, if $A$ is unital) is {\em full} in $A\rtimes_\beta H$, i.e., not contained in any proper closed two-sided ideal in $A\rtimes_\beta H$. 
Of course, this always holds if the crossed product $A\rtimes_\beta H$ is simple. The following result is then a direct consequence of Izumi's \cite[Corollary 6.6]{Izumi}. 

\begin{theorem}[Izumi]\label{thm-Izumi}
Let $\beta \colon  H\to \Aut(A)$ be an outer action of a finite group $H$ on a unital $C^*$-algebra $A$.
Then the inclusion 
$A^{H,\beta}\subseteq A$ is $C^*$-irreducible, and the intermediate algebras of the inclusion  are precisely the
 fixed-point algebras  $A^{L,\beta}$ for the subgroups $L\subseteq H$.
 \end{theorem}

\noindent
The following lemma is a modification of \cite[Lemma 3.2]{Kishimoto} by Kishimoto. We are grateful to Masaki Izumi for pointing out to us a modification 
of our original argument which assumed, in addition to the assumptions given in the lemma, that $\alpha_j$ commutes with $\beta_t$,  for all $1\leq j\leq n$ and $t\in H$. 

\begin{lemma} \label{lm:OP}
Let $A$ be a unital  simple \Cs, let $\beta \colon H \to \Aut(A)$ be an action of a finite group $H$ on $A$. Let $\alpha_1, \dots, \alpha_n$ be  automorphisms  of $A$,  and let $a_1, \dots, a_n \in A$ and  $\ep >0$ be given. Suppose that  $\alpha_j \circ \beta_t$ is outer on $A$, for all 
$1\leq j\leq n$ and for all $t \in H$.

Then there exists a positive element $h \in A^H$ with $\|h\|=1$ such that $\|ha_j\alpha_j(h)\| \le \ep$, for all $j=1,\dots,n$.
\end{lemma}

\begin{proof}  First observe that $\alpha_j \circ \beta_t$ is outer for all $t\in H$ implies that 
$\beta_{s^{-1}}\circ \alpha_j\circ \beta_t$ is outer as well, for all $s,t\in H$, which follows from the fact that the conjugate of an outer automorphism by an arbitrary automorphism remains outer.  

It follows then  from  \cite[Lemma 3.2]{Kishimoto}
 that there exists a positive element $h_0 \in A$ with $\|h_0\| = 1$ and 
 $$\|h_0\beta_{s^{-1}}(a_j) (\beta_{s^{-1}}\circ \alpha_j \circ \beta_t)(h_0)\| \le \ep|H|^{-2}, \quad s,t \in H, \; 1 \le j \le n.$$
Applying the automorphism $\beta_s$ to the inequality above, we obtain that 
$\|\beta_s(h_0)a_j \alpha_j(\beta_t(h_0))\| \le \ep|H|^{-2}$, for all $s,t \in H$ and for all $j=1,2, \dots, n$. Set $h_1 = |H|^{-1} \sum_{s \in H} \beta_s(h_0)$. Then $h_1$ is a positive element in $A^H$, and 
    $$\|h_1a_j\alpha_j(h_1)\| \le |H|^{-2}  \sum_{s,t\in H} \|\beta_s(h_0)a_j \alpha_j(\beta_t(h_0))\| \le \ep |H|^{-2}.$$
Since $\|h_1\| \ge |H|^{-1} \|h_0\| = |H|^{-1}$, it follows that $h := \|h_1\|^{-1} h_1$ has the desired properties.
\end{proof}

\noindent
We proceed to state our first main result characterizing when inclusions of the form $A^{H,\beta} \subseteq A \rtimes_{\alpha,r} G$ are $C^*$-irreducible. Thanks to some very helpful comments by Izumi we can now state this theorem in a stronger form than in a previous version of this paper, where it was assumed that the actions
$\alpha$ and $\beta$ commute and that the  group  $H$ is abelian.

\begin{theorem}\label{thm:fixed}
Let $A$ be a unital, simple $C^*$-algebra and let 
$\alpha\colon G\to \Aut(A)$ and $\beta\colon H\to \Aut(A)$ be actions of a discrete group $G$ and a finite group  $H$.
 The following are equivalent:
\begin{enumerate}
\item $A^{H,\beta} \subseteq A \rtimes_{\alpha,r} G$ is $C^*$-irreducible,
\item  $(A^{H,\beta})' \cap (A \rtimes_{\alpha,r} G) = \CC$,
\item the automorphisms $\alpha_g\circ \beta_t$ are outer for all $(e_G, e_H)\neq (g,t)\in G\times H$.
\end{enumerate}
   \end{theorem}
   
   \begin{proof}  
 (i) $\Rightarrow$ (ii) follows from  \cite[Remark 3.8]{Rordam}. 
 
  (ii) $\Rightarrow$ (iii). Suppose that $\alpha_g\circ \beta_t$ is inner for some $(e_G, e_H)\neq (g,t)\in G\times H$. Then there is a unitary $u \in A$ such that $\beta_t = \alpha_{g^{-1}} \circ \Ad u = \Ad u_{g-1}u$ (where $g \mapsto u_g \in A \rtimes_{\alpha,r} G$ is the  unitary implementation of $\alpha$). Hence $u_{g^{-1}}u \in ({A^{H}})' \cap  (A \rtimes_{\alpha,r} G)$, and $u_{g^{-1}}u \notin \CC$ since $u$ belongs to $A$ and $u_{g^{-1}}$ does not.

 (iii) $\Rightarrow$ (i).   Let $x$ be a non-zero positive element in $A \rtimes_{\alpha,r} G$. We show that $x$ is full relatively to $A^H$ in the sense of \cite[Definition 3.4]{Rordam}.  It follows then from \cite[Proposition 3.7]{Rordam} that $A^H\subseteq A\rtimes_{\alpha,r}G$ is $C^*$-irreducible. 
 
 Let $E \colon A \rtimes_{\alpha,r} G \to A$ be the canonical conditional expectation. Then  $E(x) \in A$ is non-zero and positive.
Since $A^H \subseteq A$ is $C^*$-irreducible by Theorem~\ref{thm-Izumi} (Izumi), 
 it follows from \cite[Proposition 3.7 and Lemma 3.5]{Rordam} that there exist $b_1, \dots, b_n \in A^H$ such that $1_{A^H} \le \sum_{j=1}^n b_j^* E(x) b_j = \sum_{j=1}^n E(b_j^* x b_j)$.  Upon replacing $x$ with the non-zero positive element $\sum_{j=1}^n b_j^*xb_j$,  we may therefore assume that $E(x) \ge 1_{A^H}$.
 
Let $0 < \ep < 1$ be given. Choose $y = \sum_{g \in G} a_g u_g \in A \rtimes_{\alg} G$ such that $\|x-y\| < \ep/3$.
 By Lemma~\ref{lm:OP} we can find a positive element $h \in A^H$ with $\|h\| = 1$ such that $\|h(y-E(y))h\| \le \ep/3$. This implies that $\|h(x-E(x))h\| \le \ep$. Note that
 $$hxh \ge hE(x)h-\ep \! \cdot \! 1_{A^H} \ge h^2 - \ep \! \cdot \! 1_{A^H},$$
 so $h^2 xh^2 \ge h^4 - \ep h^2$. Let $\varphi \colon [0,1] \to \RR^+$ be a continuous function which vanishes on $[0,\sqrt{\ep}]$ and which is non-zero on $(\sqrt{\ep},1]$. Then $d:= \varphi(h)(h^4 - \ep h^2)\varphi(h)$ is non-zero and $\varphi(h)h^2xh^2\varphi(h) \ge d > 0$. By simplicity of $A^H$, which follows from outerness of $\beta$, cf.\ the comments below \eqref{eq:Rosenberg}, there exist $b_1, \dots, b_n \in A^H$ such that $\sum_{j=1}^n b_j^*db_j = 1_{A^H}$. It follows that 
$$\sum_{j=1}^n b_j^*\varphi(h)h^2xh^2\varphi(h)   b_j \ge \sum_{j=1}^n b_j^*db_j = 1_{A^H},$$
which proves that $x$ is full relatively to $A^H$.
\end{proof}

\begin{remark} It follows from Izumi's \cite[Theorem 3.3]{Izumi} that an inclusion $B \subseteq A$  of simple unital \Cs s with a conditional expectation $E \colon A \to B$ of finite index is $C^*$-irreducible if (and only if) it is irreducible (i.e., $A \cap B' = \CC$). The inclusions $A^{H,\beta} \subseteq A\rtimes_{\alpha,r} G$ considered in Theorem~\ref{thm:fixed} do have finite index with respect to the composition of the canonical conditional expectations $E_1 \colon A\rtimes_{\alpha,r} G \to A$ and $E_2 \colon A \to A^{H,\beta}$ \emph{provided that $G$ is finite}. Hence the implication (ii) $\Rightarrow$ (i) of Theorem~\ref{thm:fixed} is a consequence of Izumi's theorem when $G$ is finite. Note that our proof of Theorem~\ref{thm:fixed}  does not factor through Izumi's theorem. 
\end{remark}

\begin{remark} Condition (iii) of  Theorem~\ref{thm:fixed} is equivalent to saying that the actions $\alpha\colon G\to \Aut(A)$ and $\beta\colon H\to \Aut(A)$  are outer, so that  $G$ and $H$ may be identified with subgroups of $\mathrm{Out}(A)$, the outer automorphisms on $A$, and that $G$ and $H$ intersect trivially in $\mathrm{Out}(A)$.  This condition is identical with the condition in \cite[Corollary 4.1 (i)]{BH1996} of Bisch and Haagerup ensuring irreducibility of an inclusion $P^H \subseteq P \rtimes G$ of II$_1$ factors arising from finite groups $G$ and $H$ acting outerly on a II$_1$ factor $P$. 
\end{remark}

\section{A Galois correspondence for the intermediate subalgebras}\label{sec-Galois}

In this section we shall establish a Galois type classification of the intermediate subalgebras of the inclusions considered in Theorem~\ref{thm:fixed} under the additional assumptions that the two actions $\alpha$ and $\beta$ commute and that $H$ is abelian. 

Let us first recall that  if $\alpha \colon G\to \Aut(A)$ and $\beta \colon H\to \Aut(A)$ are outer actions on a simple unital $C^*$-algebra $A$ with $G$ discrete and $H$ finite, then the intermediate algebras of the inclusions $A^{H,\beta}\subseteq A$ and 
$A\subseteq A\rtimes_{\alpha, r}G$ are in one-to-one correspondence to subgroups $L\subseteq H$ and $K \subseteq G$ by taking the fixed-point algebras $A^{L,\beta}$ and the crossed products $A\rtimes_{\alpha, r} K$, respectively, as shown by Izumi, \cite{Izumi}, and Cameron-Smith, \cite{Cameron-Smith}.

At present time it  is not clear to us how one can describe all intermediate algebras of an inclusion $A^{H,\beta}\subseteq A\rtimes_{\alpha,r}G$  in the general setting of Theorem~\ref{thm:fixed}, but we can give a satisfactory answer in the case where $H$ is abelian and the actions $\alpha$ and $\beta$ commute. 
Note that in the abelian case there is a bijection between subgroups $L$ of $H$ and subgroups of the Pontryagin dual $\widehat{H}=\Hom(H,\TT)$ 
given by $L\mapsto L^{\perp}$, where
\begin{equation} \label{eq:L-hat}
L^{\perp} :=\{x\in \widehat{H}: \langle \ell,x\rangle=1 \; \text{for all} \;  \ell \in L\}.
\end{equation}

Suppose now that $\alpha \colon G\to\Aut(A)$ and $\beta \colon H\to\Aut(A)$ are \emph{commuting} actions of discrete groups $G$ and $H$ on a simple \Cs{} $A$. Then we get an action
$$\alpha\times \beta \colon G\times H\to \Aut(A); \qquad (\alpha\times\beta)_{(g,h)}:=\alpha_g \circ \beta_h,\quad (g,h)\in G\times H.$$

We shall more than once use the fact that if $\alpha$ and $\beta$ are commuting actions as above, then $\beta$ extends naturally to an action $\tilde{\beta}$ on $A \rtimes_{\alpha, r} G$  given, for $h\in H$ and  $\sum_{g\in G}a_g u_g\in A\rtimes_{\alpha,\alg}G$, by 
\begin{equation*}\label{eq-tildeaction}
\tilde{\beta}_h(\sum_{g \in G} a_g u_g) =  \sum_{g\in G}\beta_h(a_g)u_g.
\end{equation*}

The following lemma is well-known to experts (e.g., see \cite[Lemma 2.9]{EE}, where a more general result is shown for full crossed products). 
For completeness we include the easy proof.

\begin{lemma}\label{lem-crossed}
Suppose that $\alpha\times\beta \colon G\times H\to \Aut(A)$ is an action of the discrete product group $G\times H$, as above, with $H$ is finite. Suppose further that $\beta \colon H\to \Aut(A)$ is saturated. Then the following hold:
\begin{enumerate}
\item The fixed-point algebra $A^{H,\beta}$ is a $G$-invariant subalgebra of $A$, and $\alpha$ therefore
restricts to a well-defined action $\alpha^H\colon G\to \Aut(A^{H,\beta})$;
\item  the natural extension of $\beta$ to  $\tilde\beta \colon  H\to \Aut(A\rtimes_{\alpha,r}G)$ is also saturated; 
\item  the canonical inclusion $A^{H,\beta}\rtimes_{\alpha^H,r}G\into A\rtimes_{\alpha,r}G$ co-restricts to an isomorphism
$$ A^{H,\beta}\rtimes_{\alpha^H,r}G\cong (A\rtimes_{\alpha,r}G)^{H,\tilde\beta}.$$
\end{enumerate}
\end{lemma}
\begin{proof} The first assertion is a direct consequence of the fact that $\alpha$ and $\beta$ commute. 
For the proof of (ii) we first observe that the canonical inclusion 
$$A\rtimes_\beta H \into (A\rtimes_\beta H)\rtimes_{\tilde\alpha,r}G\cong (A\rtimes_{\alpha,r}G)\rtimes_{\tilde\beta}H$$
maps the projection $p^{\beta}\in M(A\rtimes_{\beta}H)$ to the projection 
$p^{\tilde\beta}$ in the multiplier algebra $M((A\rtimes_{\alpha,r}G)\rtimes_{\tilde\beta}H)$.
Since $p^\beta$ is full in $A\rtimes_{\beta}H$ it follows that
\begin{align*}
(A\rtimes_{\alpha,r}G)\rtimes_{\tilde\beta}H&=(A\rtimes_{\beta}H)\rtimes_{\tilde\alpha,r}G\\
&\cong 
\overline{\big((A\rtimes_\beta H)p^{\beta}(A\rtimes_\beta H)\big)\rtimes_{\tilde\alpha, r}G}\\
&=\overline{\big((A\rtimes_\beta H)\rtimes_{\tilde\alpha,r}G\big)p^{\beta}\big((A\rtimes_\beta H)\rtimes_{\tilde\alpha,r} G\big)}\\
&=\overline{\big((A\rtimes_{\alpha,r}G)\rtimes_{\tilde\beta}H\big)p^{\tilde\beta} \big((A\rtimes_{\alpha,r}G)\rtimes_{\tilde\beta}H\big)}.
\end{align*}
Hence $p^{\tilde\beta}$ is full in $(A\rtimes_{\alpha,r}G)\rtimes_{\tilde\beta}H$ which proves (ii). 
The proof of  (iii) then follows from 
\begin{align*}
(A\rtimes_{\alpha,r}G)^{H,\tilde\beta}&=p^{\tilde{\beta}}\big((A\rtimes_{\alpha,r}G)\rtimes_{\tilde\beta}H\big)p^{\tilde\beta}\\
&=p^{{\beta}}\big((A\rtimes_\beta H)\rtimes_{\tilde\alpha,r}G\big)p^{\beta}\\
&=\big(p^{\beta}(A\rtimes_{\beta}H)p^{\beta}\big)\rtimes_{\tilde\alpha, r}G\\
&=A^{H,\beta}\rtimes_{\alpha,r}G,
\end{align*}
where the first and the last isomorphism in the above computation follow from Rosenberg's equation (\ref{eq:Rosenberg}).
\end{proof}

\noindent
Using the above observation, we can now prove:

\begin{proposition}\label{prop-outer} Let $\alpha$ and $\beta$ be commuting actions of discrete groups $G$ and $H$ on a simple $C^*$-algebra $A$, with $H$ finite, as above. Suppose further that $\alpha \times \beta \colon G \times H \to \Aut(A)$ is outer. Then the restricted  action 
$\alpha^H \colon G\to\Aut(A^{H,\beta})$ on the fixed-point algebra $A^{H,\beta}$ is outer.
\end{proposition}

 \begin{proof} Let $\alpha\times \beta \colon G\times H\to\Aut(A)$ be as  above. Since $A$ is simple and $\beta$ is outer, it follows from Kishimoto's theorem that $A\rtimes_\beta H$ is simple as well. Hence $\beta$ is saturated 
 and $A^{H,\beta}$ is a full corner of $A\rtimes_{\beta}H$ by the full projection $p^{\beta}$. Since full corners of simple $C^*$-algebras are simple, it follows that $A^{H,\beta}$ is  simple. 
 
 Thus, by Lemma \ref{lem-outer}, it suffices to show that for every  subgroup $M\subseteq G$ the crossed product 
$A^{H,\beta}\rtimes_{\alpha^H,r} M$ is simple. But it follows from  Lemma \ref{lem-crossed} that 
 $A^{H,\beta}\rtimes_{\alpha^H,r} M= (A\rtimes_{\alpha,r}M)^{H,\tilde\beta}$ which is a full corner of 
 $(A\rtimes_{\alpha,r}M)\times_{\tilde\beta}H\cong A\rtimes_{\alpha\times\beta,r}(M\times H)$.
 But the latter is simple, again by Kishimoto's theorem. 
 \end{proof}

%
 
 \noindent
We shall also need the lemma below. Let $\beta \colon H\to \Aut(A)$ be an action of a {\em discrete abelian} group 
 $H$ on a $C^*$-algebra $A$. The dual action $\widehat{\beta}\colon \widehat{H} \to \Aut(A\rtimes_\beta H)$ 
is for $x\in \widehat{H}$ and  $b=\sum_{h\in H} a_hu_h\in A\rtimes_{\beta,\alg}H$  given  by 
  $$\widehat\beta_x(b)=\sum_{h\in H} \overline{\langle h,x\rangle} \, a_h u_h.$$
Since $\widehat{H}$ is a compact abelian group, the subgroup $L^{\perp}$ of $\widehat{H}$, defined in \eqref{eq:L-hat}, associated with a subgroup $L$ of $H$, is compact as well.

 \begin{lemma}\label{lem-fixed}
 Suppose that $\beta \colon H\to \Aut(A)$ is an action of a discrete abelian group on a $C^*$-algebra $A$ and let $L$ be a subgroup of $H$.
 Then
 $$A\rtimes_\beta L=(A\rtimes_\beta H)^{L^\perp,\widehat{\beta}},$$
when $A\rtimes_\beta L$ is viewed as a subalgebra of $A\rtimes_\beta H$.
 \end{lemma}
 \begin{proof} Let $b=\sum_{l\in L} a_l u_l\in A\rtimes_{\alg,\beta} L$. Then 
$$
 \widehat{\beta}_x(b)=\sum_{l\in L} \overline{\langle l,x\rangle} \,  a_l  u_l=\sum_{l\in L}a_l  u_l=b,
$$
for all $x\in L^{\perp}$, so $b$ lies in $(A\rtimes_\beta H)^{L^\perp}$. This proves that $A\rtimes_\beta L \subseteq (A\rtimes_\beta H)^{L^\perp}$.

To prove the converse inclusion we make use  of the conditional expectation $E \colon A  \rtimes_\beta H \to A \rtimes_\beta L$ given by $E(b)=\int_{L^{\perp}} \widehat{\beta}_x(b)\, dx$, where the integral is with respect to the normalized Haar measure. To see that $E$ indeed maps 
$A  \rtimes_\beta H$ onto $A \rtimes_\beta L$, note first that
\begin{equation}\label{eq-Lperp}
\int_{L^{\perp}}\langle h,x\rangle\, dx= \begin{cases} 1, & \text{for} \; \;   h\in L, \\ 0, & \text{for} \; \;  h \in H \setminus L. \end{cases}
\end{equation}
Hence, for $b= \sum_{h\in H} a_h u_h\in A\rtimes_{\beta, \alg}H$, we have
$$
E(b)=\int_{L^{\perp}} \widehat{\beta}_x(b)\, dx
=\int_{L^{\perp}}\sum_{h\in H}   \overline{\langle h,x\rangle}  \,  a_h u_h\, dx = \sum_{l\in L} a_l u_l\in A\rtimes_\beta L.$$
This shows that the range of $E$ is contained in $A \rtimes_\beta L$ and that $E$ is the identity on $A \rtimes_\beta L$. Now, since $E(b) = b$, whenever $b \in (A\rtimes_\beta H)^{L^\perp}$, we are done.
 \end{proof}
  
\noindent We now provide an elaboration of the observation by Rosenberg stated in \eqref{eq:Rosenberg} relating the fixed-point algebra to a crossed product. Two inclusions $B_1 \subseteq A_1$ and $B_2 \subseteq A_2$ of \Cs s are said to be isomorphic if there is a $^*$-isomorphism $\phi \colon A_1 \to A_2$ with $\phi(B_1) = B_2$. Clearly, if $B_1 \subseteq A_1$ and $B_2 \subseteq A_2$ are isomorphic, and if one of the inclusions is $C^*$-irreducible, then so is the other.

\begin{proposition} \label{prop:morita}
Let $\beta$ be an action of a finite abelian group $H$ on a   \Cs{} $A$. Then, with $p^\beta \in M(A \rtimes_\beta H)$ as defined above \eqref{eq:Rosenberg}, there is an isomorphism $\psi \colon A \to  p^\beta(A \rtimes_\beta H \rtimes_{\widehat{\beta}} \widehat{H})p^\beta$ satisfying $\psi(A^{H,\beta}) = p^\beta(A \rtimes_\beta H)p^\beta$, thus implementing an isomorphism between 
the two inclusions
$$A^{H,\beta} \subseteq A \quad \text{and} \quad 
p^\beta(A \rtimes_\beta H)p^\beta \subseteq 
p^\beta(A \rtimes_\beta H \rtimes_{\widehat{\beta}} \widehat{H})p^\beta.$$
Moreover, for each subgroup $L\subseteq H$ we have $\psi(A^{L,\beta}) = p^{\beta}(A\rtimes_\beta H\rtimes_{\widehat{\beta}} L^{\perp})p^{\beta}$, where $L^\perp \subseteq \widehat{H}$ is the annihilator defined above 
Lemma~\ref{lem-fixed}. 
\end{proposition}

\begin{proof} Let  $u \colon H \to UM(A \rtimes_\beta H)$ and $\widehat{u} \colon \widehat{H} \to UM(A \rtimes_\beta H \rtimes_{\widehat{\beta}} \widehat{H})$
denote the canonical  representations implementing $\beta$ and $\widehat{\beta}$, respectively.
Let $\langle \, \cdot \, , \, \cdot \, \rangle \colon H \times \widehat{H} \to \TT$ denote the natural pairing between $H$ and $\widehat{H}$ as in Remark \ref{rem-outer}. 

By the definition of the dual action, $\widehat{u}_x \in A' \cap M(A \rtimes_\beta H \rtimes_{\widehat{\beta}} \widehat{H})$, for all $x \in \widehat{H}$, and $\widehat{u}_xu_g\widehat{u}_x^* =  \overline{\langle g,x \rangle} \, u_g$, for all $g \in H$ and $x \in \widehat{H}$. 

For each $g \in H$ and $x \in \widehat{H}$ set
$$p_x = \frac{1}{|H|} \sum_{g \in H}  \overline{\langle g,x \rangle} u_g, \qquad  q_g = \frac{1}{|H|} \sum_{x \in \widehat{H}}  \langle g,x \rangle  \widehat{u}_x.$$
(Note that  $|H|=|\widehat{H}|$.)
In the notation used above \eqref{eq:Rosenberg}, $p_e = p^\beta$ and $q_e = p^{\widehat{\beta}}$ (where $e$ denotes the neutral element in both groups). By definition of the dual action and the fact that $\widehat{u}$ implements $\widehat{\beta}$, it follows that
$$\widehat{u}_xu_g\widehat{u}_x^*=\widehat\beta_x(u_g)=\overline{\langle g,x\rangle} u_g, \qquad
u_g\widehat{u}_xu_g^*=u_g\widehat{u}_xu_{g^{-1}}\widehat{u}_x^*\widehat{u}_x=
\langle g,x\rangle \, \widehat{u}_x,$$
for all $g\in H, x\in \widehat{H}$. Together with  a variant of equation \eqref{eq-Lperp} it is then straightforward to verify that 
\begin{equation}\label{eq-uhatu}
1 = \sum_{g \in H} q_g = \sum_{x \in \widehat{H}} p_x, \qquad \widehat{u}_xp_e\widehat{u}_x^* = p_x, \qquad u_g q_e u_g^* = q_g,
\end{equation} 
for all $g \in H$ and $x \in \widehat{H}$.

Recall from  Lemma \ref{lem-fixed} that $A=(A\rtimes_\beta H)^{\widehat{H}}$. 
By Rosenberg's result, cf.\ \eqref{eq:Rosenberg}, we have $^*$-isomorphisms
$$\varphi \colon  A^H \to p_e(A \rtimes_\beta H)p_e, \qquad \psi_0 \colon A \to q_e(A \rtimes_\beta H \rtimes_{\widehat{\beta}} \widehat{H})q_e,$$
given by $\varphi(b) = b p_e = |H|^{-1}\sum_{g \in H} b u_g$ and $\psi_0(a) = a q_e = |{H}|^{-1} \sum_{x \in \widehat{H}} a \widehat{u}_x$, for $b \in A^H$ and $a \in A$.

Now, by Takai duality, the two projections $p_e$ and $q_e$ are equivalent in the \Cs{} generated by $\{u_g\}_{g \in H} \cup \{\widehat{u}_x\}_{x \in \widehat{H}}$ (since this $C^*$-algebra is isomorphic to $M_{|H|}(\CC)$ and $p_e$ and $q_e$ are minimal projections herein). We can also see this directly as follows: For $x \in \widehat{H}$ we have $p_e\widehat{u}_xp_e = p_ep_x\widehat{u}_x = \delta_{e,x} \, p_e$, so $p_eq_ep_e = |{H}|^{-1} p_e$. Similarly, $q_ep_eq_e = |H|^{-1} q_e$. Set $z =  |H|^{1/2} p_eq_e$. Then $z^*z = q_e$ and $zz^*=p_e$. Note that $z$ commutes with $A^H$. 
Define a $^*$-isomorphism
\begin{equation}\label{eq-psi}
\psi \colon A \to p_e(A \rtimes_\beta H \rtimes_{\widehat{\beta}} \widehat{H})p_e,\quad \psi(a) = z\psi_0(a)z^* \; 
(=|H|\, p_eaq_ep_e), \; \; a \in A.
\end{equation}
For $b \in A^H$ we have $\psi(b) = z(bq_e)z^* = bzq_ez^*=bp_e = \varphi(b)$. Hence $\psi(A^H) = \varphi(A^H) = p_e(A \rtimes_\beta H)p_e$, as desired.

\medskip
Let $L\subseteq H$ be a subgroup. We check that $\psi(A^L)=p_e(A\rtimes_\beta H\rtimes_{\widehat{\beta}}L^\perp)p_e$, where we view $A\rtimes_\beta H\rtimes_{\widehat{\beta}}L^\perp$ as a subalgebra of 
$A\rtimes_\beta H\rtimes_{\widehat{\beta}}\widehat{H}$ in the canonical way. Recall from Lemma \ref{lem-fixed}, applied to $\widehat{\beta}$ via the isomorphism $H\cong \widehat{\widehat{H}}$, which sends $g\in H$ to $\big(x\mapsto \langle g,x\rangle\big)\in \widehat{\widehat{H}}$, that 
$$A\rtimes_\beta H\rtimes_{\widehat{\beta}}L^\perp=(A\rtimes_\beta H\rtimes_{\widehat{\beta}}\widehat{H})^{L,\widehat{\widehat{\beta}}}.$$
Since $p_e\in A\rtimes_\beta H$ is fixed by $\widehat{\widehat{\beta}}$, we see that $\widehat{\widehat{\beta}}$ restricts to an action 
on $p_e(A\rtimes_\beta H\rtimes_{\widehat{\beta}}\widehat{H})p_e$. So the result will follow if we can show that 
the isomorphism $\psi \colon A\to p_e(A\rtimes_\beta H\rtimes_{\widehat{\beta}}\widehat{H})p_e$ is $\beta$-$\widehat{\widehat\beta}$ equivariant. To this end observe first that for all $g\in H$ we have
$$\widehat{\widehat{\beta}}_g(q_e)=\frac{1}{ |{H}|} \sum_{x\in \widehat{H}} \overline{\langle g, x\rangle} \,  \widehat{u}_x= q_{g^{-1}}=u_g^*q_eu_g.$$
Using this, and the fact that $p_e$ is fixed by $\widehat{\widehat{\beta}}$ we get for all $a\in A$ and $g\in H$:
\begin{align*}
\widehat{\widehat{\beta}}_g(\psi(a))&\stackrel{(\ref{eq-psi})}{=}|H| \, \widehat{\widehat{\beta}}_g(p_eaq_ep_e)
=|H| \, p_ea\widehat{\widehat{\beta}}_g(q_e)p_e
=|H| \, p_ea u_g^*q_eu_gp_e\\
&=|H| \, p_eu_g^*\beta_g(a)q_eu_gp_e \overset{(*)}{=} |H| \, p_e\beta_g(a)q_ep_e=\psi(\beta_g(a)),
\end{align*}
where  at ($*$) we have used the fact that $p_eu_g^*=u_gp_e=p_e$ for all $g\in H$, which follows easily from the definition of $p_e$. 
This finishes the proof.
\end{proof}

\begin{lemma} \label{lem:compression}
 Let $B \subseteq A$ be a  unital inclusion of  \Cs s and let $p \in B$ be a projection. If $B \subseteq A$ is $C^*$-irreducible, then so is $pBp \subseteq pAp$. Conversely, if $p$ is full in $B$ and if $pBp\subseteq pAp$ is $C^*$-irreducible, then
 $B\subseteq A$ is $C^*$-irreducible as well. Moreover, in this case the assignment $D\mapsto pDp$ gives a bijective correspondence between the intermediate \Cs s of $B\subseteq A$ and those  of $pBp\subseteq pAp$.
\end{lemma}
\begin{proof} Assume first that $B\subseteq A$ is $C^*$-irreducible.
Let $pBp\subseteq C\subseteq pAp$ be an intermediate \Cs, and set $D=C^*(B\cup C)$. 
Then $B\subseteq D\subseteq A$, so $D$ is simple. Moreover, $C=pDp$, so $C$ is a corner of the simple $C^*$-algebra $D$, and is hence simple as well.

Suppose now that $p$ is full and that $pBp\subseteq pAp$ is $C^*$-irreducible. If $B\subseteq D\subseteq A$ is any intermediate 
$C^*$-algebra, then $pBp\subseteq pDp\subseteq pAp$, and hence $pDp$ is simple. Since $p$ is full in $B$, it follows that $p$ is also full in $D$, and this implies that $D$ is simple.

As for the last claim, we remarked above that the assignment $C \mapsto C^*(B \cup C)$ gives a map from intermediate \Cs s of the inclusion $pBp \subseteq pAp$ to intermediate \Cs s of the inclusion $B \subseteq A$, which is a left-inverse of the assignment $D \mapsto pDp$, i.e., $pC^*(B\cup C)p = C$, for any $pBp \subseteq C \subseteq pAp$. If $p$ is full in $B$, then it is also a right-inverse, i.e., $D = C^*(B\cup pDp)$, for any $B \subseteq D \subseteq A$. Indeed, $1= 1_B = \sum_{j=1}^n b_j^*pb_j$, for some $b_1, \dots, b_n \in B$ by fullness of $p$ in $B$. Hence, for each $d \in D$, we have $d = 1 \! \cdot \! d \! \cdot \! 1 = \sum_{i,j=1}^n b_i^*pb_idb_jpb_j^*$, which belongs to $C^*(B\cup pDp)$, since $pb_idb_jp \in pDp$, for all $i,j$. 
\end{proof}

\noindent
We are now ready to give a Galois type classification of the intermediate subalgebras of (some of) the inclusion $A^{H,\beta} \subseteq A \rtimes_{\alpha,r} G$ considered in Theorem~\ref{thm:fixed}.

\begin{theorem}\label{thm-intermediate} Suppose that $\alpha \colon G \to \Aut(A)$ and $\beta \colon H \to \Aut(A)$ are commuting actions of a discrete group $G$ and a finite abelian group $H$ on a unital simple \Cs{} $A$. 
\begin{enumerate}
\item The inclusion $A^{H,\beta} \subseteq A \rtimes_{\alpha,r} G$ is isomorphic to the inclusion
\begin{equation} \label{eq:G}
 p^\beta(A \rtimes_\beta H)p^\beta \subseteq  p^\beta(A \rtimes_\beta H \rtimes_{\widehat{\beta}} \widehat{H} \rtimes_{{\tilde{\alpha}},r} G)p^\beta
 \end{equation}
 where $p^\beta$ is as defined in \eqref{eq:p-beta}, and where $\tilde{\alpha} \colon G \to \Aut(A \rtimes_\beta H \rtimes_{\widehat{\beta}} \widehat{H})$ is the extension of $\alpha$, cf.\ the explanation above Lemma~\ref{lem-crossed}. 
\item There is a one-to-one correspondence between subgroups $L\subseteq \widehat{H}\times G$ and intermediate algebras of the inclusion in \eqref{eq:G} given by sending $L$ to
$p^{\beta}(A\rtimes_\beta H)p^{\beta}\rtimes_{\widehat{\beta}\times\widetilde{\alpha},r} L = p^{\beta}(A\rtimes_\beta H \rtimes_{\widehat{\beta}\times\widetilde{\alpha},r} L)p^{\beta}$.
\item There is a one-to-one correspondence between subgroups  of $\widehat{H}\times G$ and intermediate algebras of the inclusion $A^{H,\beta} \subseteq A \rtimes_{\alpha,r} G$.

In particular, if  $L=L_1\times L_2$ is a product of subgroups $L_1\subseteq \widehat{H}$ and $L_2\subseteq G$, then the corresponding intermediate algebra $A^{H,\beta} \subseteq D \subseteq A \rtimes_{\alpha,r} G$ is $D=A^{{L_1^\perp},\beta}\rtimes_{\alpha,r}L_2$, with $L_1^\perp$ the annihilator of $L_1$ in $H$, cf.\ \eqref{eq:L-hat}.
\end{enumerate}
\end{theorem}
\begin{proof} (i). It was shown in Proposition~\ref{prop:morita} that the inclusion $A^H \subseteq A$ is isomorphic to the inclusion $p^\beta(A \rtimes_\beta H)p^\beta \subseteq 
p^\beta(A \rtimes_\beta H \rtimes_{\widehat{\beta}} \widehat{H})p^\beta$ via the
  $^*$-isomorphism 
$$\psi \colon A \to p^\beta(A \rtimes_\beta H \rtimes_{\widehat{\beta}} \widehat{H})p^\beta,$$
defined in \eqref{eq-psi},
that maps $A^H$ onto  $p^\beta(A \rtimes_\beta H)p^\beta$.   The isomorphism $\psi$ is easily seen to be $\alpha$-$\tilde{\alpha}$ equivariant. Hence it
extends naturally to a $^*$-isomorphism $\bar{\psi} \colon A \rtimes_{\alpha,r} G \to p^\beta(A \rtimes_\beta H \rtimes_{\widehat{\beta}} \widehat{H} )p^\beta \rtimes_{{\tilde{\alpha}},r} G$. The algebra $p^\beta(A \rtimes_\beta H \rtimes_{\widehat{\beta}} \widehat{H} )p^\beta \rtimes_{{\tilde{\alpha}},r} G$ is equal to $p^\beta(A \rtimes_\beta H \rtimes_{\widehat{\beta}} \widehat{H} \rtimes_{{\tilde{\alpha}},r} G)p^\beta$ because ${\tilde{\alpha}}_g(p^\beta) = p^\beta$ for all $g \in G$ by the definition of $\tilde{\alpha}$. The $^*$-isomorphism $\bar{\psi}$ therefore implements the desired isomorphism of the two inclusions. 

(ii).
Since $A^H \subseteq A \rtimes_{\alpha,r} G$ is $C^*$-irreducible by Theorem \ref{thm:fixed}, so is the inclusion in \eqref{eq:G}, and hence so is the inclusion 
\begin{equation} \label{eq:H}
A \rtimes_\beta H \subseteq A \rtimes_\beta H \rtimes_{\widehat{\beta}} \widehat{H} \rtimes_{{\tilde{\alpha}},r} G =
A \rtimes_\beta H \rtimes_{\widehat{\beta} \times \tilde{\alpha},r} (\widehat{H} \times  G),
\end{equation}
by Lemma~\ref{lem:compression}. It follows from \cite[Theorem 5.8]{Rordam} that $\widehat{\beta} \times \tilde{\alpha} \colon \widehat{H} \times G \to \Aut(A \rtimes_\beta H)$ is outer. 

By Lemma~\ref{lem:compression} there is a bijective correspondence between intermediate \Cs s of the inclusion in \eqref{eq:H} and intermediate \Cs s of the inclusion in \eqref{eq:G} given by  compression with $p^\beta$. Finally, by the Cameron--Smith theorem, \cite[Theorem 3.5]{Cameron-Smith}, which applies because $\widehat{\beta} \times \tilde{\alpha}$ is outer, each intermediate \Cs{} of the  inclusion in \eqref{eq:H} is of the form $$(A \rtimes_\beta H) \rtimes_{\widehat{\beta} \times {\tilde{\alpha}},r} L,$$ for some subgroup $L$ of $ \widehat{H} \times G$. This proves (ii).

(iii) follows from (i) and (ii) and, for the last claim, inspection of the isomorphism $\psi$ which implements the isomorphism of the two inclusions in (i).
\end{proof}

\section{Examples} \label{sec:examples}

\noindent
In this section we want to discuss some interesting examples of the theory as developed in the previous sections arising from group actions on the irrational rotation algebra $A_\theta$, for $\theta\in \RR\setminus \QQ$.
Recall that  $A_\theta$ is the universal 
  $C^*$-algebra generated by two unitaries $u,v$ subject to the relation
  $$vu= e^{2\pi i\theta}uv.$$
  There is an outer action   $\alpha\colon  \SL(2,\ZZ)\to \Aut(A_\theta)$ for which
  $$n=\begin{pmatrix} n_{11}& n_{12}\\ n_{21}& n_{22}\end{pmatrix} \in \SL(2,\ZZ)$$
 acts on the generators $u$, $v$ of $A_\theta$ by
  $$\alpha_n(u)= e^{2\pi i n_{11}n_{21}\theta}u^{n_{11}}v^{n_{21}}, \qquad
  \alpha_n(v)= e^{2\pi i n_{12}n_{22}\theta}u^{n_{12}}v^{n_{22}}.$$
  Up to conjugacy, there are exactly four different finite cyclic subgroups of $\SL(2,\ZZ)$ isomorphic to
the cyclic groups $\ZZ_2, \ZZ_3, \ZZ_4$, and $\ZZ_6$, generated, in that order, by the elements:
  \begin{equation}\label{eq-gen}
\begin{split}
\begin{pmatrix} -1 & 0 \\ 0 & -1 \end{pmatrix}, \quad \begin{pmatrix} 0 & 1 \\ -1 & -1 \end{pmatrix}, \quad 
 \begin{pmatrix} 0 & -1 \\ 1 & 0 \end{pmatrix}, \quad \begin{pmatrix} 0 & -1 \\ 1 & 1 \end{pmatrix}.
\end{split}
\end{equation}

The resulting crossed products $A_\theta \rtimes_\alpha \ZZ_k$, $k=2,3,4,6$, have been studied in detail in \cite{ELPW}, where it has been shown that they as well as the fixed-point algebras $A_\theta^{\ZZ_k}$, $k=2,3,4,6$, are simple AF-algebras. 
By \cite[Theorem 5.8]{Rordam}, all inclusions $A_\theta\subseteq A_\theta\rtimes_\alpha \ZZ_k$ are $C^*$-irreducible, and it follows from 
Theorem~\ref{thm-Izumi} (Izumi) that the inclusions $A_\theta^{\ZZ_k}\subseteq A_\theta$ are $C^*$-irreducible as well.
Thus we see that every $A_\theta$, with $\theta$ irrational, has a unital $C^*$-irreducible inclusion into some simple AF-algebra, and that, on the other hand, there always exist simple AF-algebras which admit a unital  $C^*$-irreducible embedding into $A_\theta$. 
But note that the composition $A_\theta^{\ZZ_k}\subseteq A_{\theta}\rtimes_\alpha \ZZ_k$ of these inclusions 
is not $C^*$-irreducible, since $(A_\theta^{G})'\cap(A_\theta\rtimes_\alpha G)\neq \CC$, as observed earlier for general actions $\alpha \colon G\to\Aut(A)$ of a finite group $G$. On the other hand, since the entire group $\SL(2,\ZZ)$ acts by outer automorphisms on $A_\theta$, condition 
(iii) of Theorem \ref{thm:fixed} is satisfied for the actions of two subgroups $F_1, F_2\subseteq \SL(2,\ZZ)$ on $A_\theta$ if and only if their intersection $F_1\cap F_2$ is trivial in $\SL(2,\ZZ)$. We therefore get:

\begin{proposition}\label{prop-ex}
Suppose that $(F_1, F_2)$ is either one of the pairs 
$$(\ZZ_2, \ZZ_3),\qquad (\ZZ_3, \ZZ_4), \qquad  (\ZZ_3, \widetilde{\ZZ}_3),$$
where $\widetilde{\ZZ}_3:=\langle R\rangle$ for some matrix $R\in \SL(2,\ZZ)$ which is a conjugate of the matrix $\left(\begin{smallmatrix} 0 & 1 \\ -1 & -1 \end{smallmatrix}\right)$  inside $\SL(2,\ZZ)$ and for which $\ZZ_3\cap \widetilde{\ZZ}_3=1$.\footnote{One can for example take $R=\left(\begin{smallmatrix} -2 & 1 \\ -3 & 1 \end{smallmatrix}\right) = S\left(\begin{smallmatrix} 0 & 1 \\ -1 & -1 \end{smallmatrix}\right)S^{-1}$, with $S = \left(\begin{smallmatrix} 1 & 1 \\ 1 & 2 \end{smallmatrix}\right)$.}
Then
$$A_\theta^{F_1}\subseteq A_\theta\rtimes F_2, \qquad A_\theta^{F_2}\subseteq A_\theta\rtimes F_1$$
are $C^*$-irreducible inclusions of AF-algebras.
\end{proposition}
\begin{proof} In all these cases we have $F_1\cap F_2=1$ in $\SL(2,\ZZ)$, so the result follows from Theorem \ref{thm:fixed}.
\end{proof}

\noindent 
Among  the  finite subgroups of $\SL(2,\ZZ)$ listed in and above \eqref{eq-gen}, the pairs $(F_1, F_2)$ listed in the proposition above are the only ones 
which satisfy item (iii) of Theorem \ref{thm:fixed}, so any other combination of subgroups $(F_1,F_2)$ will not provide $C^*$-irreducible inclusions. 

Since $A_\theta$ is not an AF-algebra, Proposition \ref{prop-ex} leads (as expected) to a negative answer to \cite[Question 6.11]{Rordam}:

\begin{corollary} \label{cor:AF-non-AF} There exist $C^*$-irreducible inclusions of AF-algebras with intermediate \Cs s that are not AF-algebras.
\end{corollary}

\noindent Of the three pairs of groups $(F_1,F_2)$ in Proposition~\ref{prop-ex} above, only the pair $(\ZZ_2,\ZZ_3)$
satisfies the additional assumptions of Theorem \ref{thm-intermediate} which gives a classification of the intermediate \Cs s.  This pair also satisfies the conditions of the following:

\begin{proposition}\label{prop-inter}
Suppose that $H$ and $G$ are finite cyclic groups of prime order $p$ and $q$, respectively, such that $p\neq q$.
Let $\alpha\times\beta:G\times H\to \Aut(A)$ be an outer action on the simple unital $C^*$-algebra $A$. 

Then $A^{H,\beta}\subseteq A\rtimes_\alpha G$ is a $C^*$-irreducible inclusion, and  
$A$ and $A^{H,\beta}\rtimes_\alpha G$ are the only (strict) intermediate  \Cs s for this inclusion.
\end{proposition}

\begin{proof} 
Since  finite cyclic groups are self-dual, it follows from the assumption on the pair  $p,q$ that 
$\widehat{H}\cong \widehat{H}\times\{e\}$ and $G\cong \{e\}\times G$ are the only non-trivial subgroups of $\widehat{H}\times G$.
Thus it follows from Theorem \ref{thm-intermediate} that $A=A^{\widehat{H}^{\perp},\beta}$ and 
$A^{H,\beta}\rtimes_\alpha G=A^{\{e\}^\perp,\beta}\rtimes_\alpha G$ are the only strict intermediate $C^*$-algebras for the inclusion 
$A^{H,\beta}\subseteq A\rtimes_\alpha G$.
\end{proof}

\begin{corollary}\label{cor-inter}
Let $\theta\in \RR\setminus\QQ$. The only strict intermediate $C^*$-algebras for the $C^*$-irreducible inclusion 
$A_\theta^{\ZZ_2,\alpha}\subseteq A_\theta\rtimes_\beta \ZZ_3$ are $A_\theta$ and $A_\theta^{\ZZ_2, \alpha}\rtimes_\beta\ZZ_3$.

Similarly, the only strict intermediate $C^*$-algebras for the $C^*$-irreducible inclusion 
$A_\theta^{\ZZ_3,\beta}\subseteq A_\theta\rtimes_\alpha \ZZ_2$ are $A_\theta$ and $A_\theta^{\ZZ_3, \beta}\rtimes_\alpha\ZZ_2$.
\end{corollary}

\noindent
Note that the intermediate algebras 
 $A_\theta^{\ZZ_2, \alpha}\rtimes_\beta\ZZ_3$ and  $A_\theta^{\ZZ_3, \beta}\rtimes_\alpha\ZZ_2$ are AF algebras.
Indeed,  it is shown in \cite{ELPW} that 
 $A_\theta\rtimes_\gamma \ZZ_6=A_\theta\rtimes_{\alpha\times\beta}(\ZZ_2\times \ZZ_3)$ 
is an  AF-algebra.
By Lemma \ref{lem-crossed} together with Rosenberg's isomorphism (\ref{eq:Rosenberg})
 it follows  that 
$$A_\theta^{\ZZ_2, \alpha}\rtimes_\beta\ZZ_3=(A_\theta\rtimes_\beta \ZZ_3)^{\ZZ_2,\alpha}$$
is a (full) corner of  $A_\theta\rtimes_\beta\ZZ_3\rtimes_{\tilde\alpha}\ZZ_2\cong A_\theta\rtimes_\gamma \ZZ_6$, and similarly for 
$A_\theta^{\ZZ_3, \beta}\rtimes_\alpha\ZZ_2$. Since corners of AF-algebras are AF, it follows that
$A_\theta^{\ZZ_2, \alpha}\rtimes_\beta\ZZ_3$ and $A_\theta^{\ZZ_3, \beta}\rtimes_\alpha\ZZ_2$ are AF-algebras.

\begin{remark}\label{rem-intermediate}
It would be very interesting also to understand the intermediate $C^*$-algebras of the  inclusions 
appearing in Proposition \ref{prop-ex}, other than the ones arising from the pair $(\ZZ_2,\ZZ_3)$.

Perhaps, the most interesting case is given by the inclusion  
 $A_\theta^{\ZZ_3}\subseteq A_\theta\rtimes \widetilde{\ZZ}_3$. 
 The only obvious intermediate \Cs{} here is $A_\theta$ itself, and it might well be that it is the only one. (By an ``obvious'' intermediate \Cs{} of an inclusion $A^H \subseteq A \rtimes_r G$, we think here of one of the form $D \rtimes_{r,\alpha} L$, where $L$ is a subgroup of $G$ and $D$ is an $L$-invariant intermediate algebra $A^{H} \subseteq D \subseteq A$.)
 If that would be true it would give us an example of a $C^*$-irreducible inclusion of two AF algebras with $A_\theta$ as the unique intermediate $C^*$-algebra. 
 
Since $\widetilde{\ZZ}_3$ is a conjugate of $\ZZ_3$ by an element of $\SL(2,\ZZ)$, the crossed product
 $A_\theta\rtimes \widetilde{\ZZ}_3$ is canonically isomorphic to the crossed product $A_\theta\rtimes {\ZZ}_3$ in which 
 $A_\theta^{\ZZ_3}$ sits as a full corner. In particular, $A_\theta^{\ZZ_3}$ and $A_\theta\rtimes \widetilde{\ZZ}_3$ are Morita equivalent AF-algebras.
 \end{remark}

\noindent
\subsection*{Actions by infinite cyclic groups.} 
Actions  on $A_\theta$ can provide further  examples of $C^*$-irreducible inclusions with interesting properties. 
For this let us consider actions of $\ZZ$ on $A_\theta$ which are given by restrictions of the action of $\SL(2,\ZZ)$ to 
infinite cyclic subgroups. These are generated by matrices $S\in \SL(2,\ZZ)$ of infinite order. Let us then write $\alpha^S$ 
for the corresponding action of $\ZZ$ an $A_\theta$. The crossed products 
$A_\theta\rtimes_{\alpha^S}\ZZ$ have been studied and classified in \cite{BCHL}. A particularly interesting example occurs 
if $\tr(S)=3$, e.g., for $S=\left(\begin{smallmatrix} 1&1\\1&2\end{smallmatrix}\right)$. 
In this case, the classification results of \cite{BCHL}  imply that $A_\theta\rtimes_{\alpha^S}\ZZ$ is actually isomorphic to $A_\theta$ itself.
Thus by \cite[Theorem 5.8]{Rordam} and \cite{Cameron-Smith} we obtain a proper $C^*$-irreducible inclusion 
$$A_\theta\subseteq A_\theta\rtimes_{\alpha^S}\ZZ\cong A_\theta.$$
By the results of Cameron and Smith in \cite[Theorem 3.5]{Cameron-Smith}, all (strict) intermediate $C^*$-algebras are of the form
$$A_\theta\rtimes_{\alpha^S}(n\ZZ)= A_{\theta}\rtimes_{\alpha^{S^n}}\ZZ, \quad n=2,3,4,\ldots.$$
 Using the results of \cite[Theorem 3.5]{BCHL}, all these intermediate algebras can be classified by their Elliott invariants,
and it turns out that they are never AF (since by \cite[Theorem 3.5]{BCHL} their $K_1$-groups never vanish) and they are usually not isomorphic to $A_\theta$.

\begin{example}\label{ex-tr3}
Let us look again at the matrix $S=\left(\begin{smallmatrix} 1&1\\1&2\end{smallmatrix}\right)$. 
Then $S$ is self-adjoint with $\tr(S)=3$. The entries of the powers of $S$ are Fibonnaci numbers: 
$$S^n = \begin{pmatrix} f_{2n-1}&f_{2n}\\f_{2n}&f_{2n+1}\end{pmatrix}, \qquad n \ge 1.$$ 
In particular,  it follows that $\tr(S^n)> 3$, for all $n\geq 2$, and hence it follows from 
 \cite[Theorems 3.5 and 3.9]{BCHL}  that the intermediate algebras 
$A_\theta\rtimes_{\alpha^{S^n}}\ZZ$ of the inclusion $A_\theta\subseteq A_{\theta}\rtimes_{\alpha^S}\ZZ\cong A_\theta$ 
are never isomorphic to $A_\theta$ and are not even irrational rotation algebras.

Indeed, using \cite[Remark 3.12]{BCHL}, we can conclude that $A_\theta\rtimes_{\alpha^{S^n}}\ZZ$ and $A_\theta\rtimes_{\alpha^{S^m}}\ZZ$ are never isomorphic if $n\neq m$, since we have
$|2-\tr(S^n)|\neq |2-\tr(S^m)|$, whenever $n,m\in \NN$ with $n\neq m$.
\end{example}

\begin{remark}
For any element $S\in \SL(2,\ZZ)$ of infinite order, the intersection $\langle S\rangle \cap F$ is trivial for any finite subgroup $F\subseteq \SL(2,\ZZ)$.
Therefore, with $S=\left(\begin{smallmatrix} 1&1\\1&2\end{smallmatrix}\right)$ as above, we get $C^*$-irreducible inclusions
$$A_\theta^F\subseteq A_\theta\rtimes_{\alpha^S}\ZZ\cong A_\theta$$
for every such subgroup $F$. In the case where $F=\ZZ_2$, which is  generated by the central element $\left(\begin{smallmatrix} -1&0\\0&-1\end{smallmatrix}\right)$, the 
actions of $F$ and $\ZZ$ commute and Theorem \ref{thm-intermediate}  gives  a description of all intermediate algebras for this inclusion. 
\end{remark}

\noindent
Another interesting consequence of this type of examples is the existence of outer actions $\beta^n$ of the cyclic groups $\ZZ_n$ on 
$A_\theta$, for all $n\in \NN$ with $n\geq 2$, such that the crossed products $A_{\theta}\rtimes_{\beta^n}\ZZ_n$ as well as the fixed-point algebras $A_\theta^{\ZZ_n,\beta^n}$ are not AF, quite contrary to the case of the actions of the finite subgroups of $ \SL(2,\ZZ)$ considered before.  For this we need

\begin{lemma}\label{lem-outer-dual}
  Suppose that $\beta \colon H\to \Aut(A)$ is an outer action of the discrete abelian group $H$ on a simple $C^*$-algebra $A$. Then, for each finite subgroup $M\subseteq \widehat{H}$,  the restriction of the dual action $\widehat{\beta} \colon \widehat{H}\to\Aut(A\rtimes_\beta H)$ to $M$ is outer as well. 
  \end{lemma}
  
  \noindent If $\widehat{H}$ is finite, or more generally if $\widehat{H}$ has no element of infinite order, then the  lemma simply says that $\widehat{\beta}$ itself also is outer, cf.\ Lemma \ref{lem-outer}.
  
  \begin{proof}
  Let $L\subseteq M\subseteq \widehat{H}$ be any subgroup of $M$ and let $L^{\perp}$ be the annihilator 
   of $L$ in $H$. Then it follows from \cite[Proposition 2.1]{Ech-indres} that 
   $(A\rtimes_\beta H)\rtimes_{\widehat\beta} L$ is Morita equivalent to $A\rtimes_\beta L^{\perp}$, which is simple by Lemma \ref{lem-outer}. Thus, since  Morita equivalence preserves simplicity, the crossed product    $(A\rtimes_\beta H)\rtimes_{\widehat\beta} L$ 
   is simple as well. Thus, it follows from Lemma \ref{lem-outer} that the restriction of $\widehat\beta$ to $M$ is by outer automorphisms. 
      \end{proof}

\begin{example}\label{ex-finiteS}
Let $S=\left(\begin{smallmatrix} 1&1\\1&2\end{smallmatrix}\right)$ as above (for most of what we do here, one could take any $S\in  \SL(2,\ZZ)$ with $\tr(S)=3$).  
Consider the dual action $\widehat{\alpha}^S \colon \TT\to \Aut(A_\theta\rtimes_{\alpha^S}\ZZ)$  of $\alpha^S$. The isomorphism 
$A_\theta\rtimes_{\alpha^S}\ZZ\cong A_\theta$ carries this to an action, say $\beta \colon \TT\to \Aut(A_\theta)$.
For each $n\in \NN$, let us identify the cyclic group $\ZZ_n$ of order $n$ with the group of all $n$th roots of unity in $\TT$, 
which is the annihilator of $n\ZZ\subseteq \ZZ$ under the identification $\TT\cong \widehat{\ZZ}$. Thus 
 $\ZZ_n$ can be identified with $(n\ZZ)^{\perp}\subseteq \TT$. It follows from 
Lemma \ref{lem-outer-dual} that the restriction of $\beta$ to $\ZZ_n$ gives an outer action, called $\beta^n$ below,  of $\ZZ_n$ on $A_\theta$. Thus, using \cite[Theorem 5.8]{Rordam} and Theorem \ref{thm:fixed}, we obtain 
$C^*$-irreducible inclusions 
$$A_\theta^{\ZZ_n,\beta^n}\subseteq A_\theta\quad\text{and}\quad A_\theta \subseteq A_\theta\rtimes_{\beta^n}\ZZ_n$$
with intermediate algebras given by $A_\theta^{\ZZ_m,\beta^m}$ and $A_\theta\rtimes_{\beta^m} \ZZ_m$,   respectively,
for all $m\in \NN$ which divide $n$. It follows then from 
Lemma \ref{lem-fixed} that 
$$A_\theta^{\ZZ_m,\beta^m}\cong A_\theta\rtimes_{\alpha^{S^m}}\ZZ.$$
So at least for $S=\left(\begin{smallmatrix} 1&1\\1&2\end{smallmatrix}\right)$, it follows from Example \ref{ex-tr3} that the sequence of \Cs s above are pairwise non-isomorphic, and that none of them are AF-algebras. 

Note, if $n,m\in \NN$ have no common divisors, then $\ZZ_n\cap \ZZ_m=\{0\}$ and Theorem \ref{thm:fixed} implies that the inclusion 
$$A_\theta^{\ZZ_n,\beta^n}\subseteq A_\theta\rtimes_{\beta^m}\ZZ_m$$
is also $C^*$-irreducible. Again, in this case Theorem \ref{thm-intermediate} allows us to compute all intermediate algebras of this inclusion.
\end{example}

\begin{question}
Let $A_\theta\subseteq A_\theta\rtimes_{\alpha^S}\ZZ\cong A_\theta$ be the $C^*$-irreducible inclusion considered in Example~\ref{ex-finiteS} above. 
By iteration we get a chain of inclusions
$$A_\theta\subseteq A_{\theta}\subseteq\cdots\subseteq A_\theta \subseteq \cdots.$$
Are all compositions in this sequence $C^*$-irreducible?

It has been shown in \cite[Remark 3.11]{BCHL} that the direct  limit of this sequence is the AF-algebra 
constructed by Effros and Shen in \cite{ES}, and into which $A_\theta$ embeds with the same ordered $K_0$-groups, as shown by Pimsner and Voiculescu in \cite{PV}.
\end{question}


\begin{thebibliography}{19}

\bibitem{Baggett-Kleppner}
L.  Baggett and A.  Kleppner,
Multiplier representations of abelian groups,
\textit{J. Functional Analysis}~\textbf{14} (1973), 299--324. 

\bibitem{BCHL}
C.  B\"{o}nicke, S.  Chakraborty, Z.  He, and H.-C.  Liao,
Isomorphism and {M}orita equivalence classes for crossed
              products of irrational rotation algebras by cyclic subgroups
              of {$SL_2(\mathbb{Z})$},
              \textit{J. Functional Analysis}~\textbf{275} (2018), 3208--3243.


\bibitem{BH1996}
D.  Bisch and U.  Haagerup,
Composition of subfactors: new examples of infinite depth
              subfactors,
              \textit{Ann. Sci. \'{E}cole Norm. Sup. (4)}~\textbf{29} (1996), no.  3, 329--383.

\bibitem{Cameron-Smith}
J.  Cameron and R. R.  Smith,
A {G}alois correspondence for reduced crossed products of
              simple {$\rm C^*$}-algebras by discrete groups,
              \textit{Canad. J. Math.}~\textbf{71} (2019), no.  5, 1103--1125. 


\bibitem{Cameron-Smith-Cor}
J.  Cameron and R. R.  Smith,
Corrigendum to: {A} {G}alois correspondence for reduced
              crossed products of simple {${\rm C}^*$}-algebras by discrete
              groups,
              \textit{Canad. J. Math.}~\textbf{72} (2020), no.  2, 557--562. 
              
 \bibitem{Combes}
 F.  Combes,
 Crossed products and {M}orita equivalence,
 \textit{Proc. London Math. Soc. (3)}~\textbf{49} (1984), no.  2, 289--306.


\bibitem{CELY}
J.  Cuntz, S.  Echterhoff, X.  Li, and G.  Yu,
\textit{{$K$}-theory for group {$C^*$}-algebras and semigroup
              {$C^*$}-algebras},
              Oberwolfach Seminars,~47, Birk\"{a}user/Springer, Cham, 2017. 
              
              
\bibitem{Ech-indres}
S.  Echterhoff,
Duality of induction and restriction for abelian twisted
              covariant systems,
              \textit{Math. Proc. Cambridge Philos. Soc.}~\textbf{116} (1994), no.  2, 301--315. 
              
              
\bibitem{EE}
S.  Echterhoff and H.  Emerson,
Structure and {$K$}-theory of crossed products by proper
              actions,
\textit{Expo. Math.},
\textbf {29}  (2011), no. (3), 300--344.  


              
\bibitem{ELPW}
S.  Echterhoff, W.  L\"{u}ck, N. C.  Phillips, and S.  Walters,
    The structure of crossed products of irrational rotation
              algebras by finite subgroups of {${\rm SL}_2(\Bbb Z)$},
  \textit{J. Reine Angew. Math.}~\textbf{639} (2010),  173--221.

\bibitem{ES}
E. G.  Effros and C.  L.  Shen,
           Approximately finite {$C^{\ast} $}-algebras and continued
              fractions,
\textit{Indiana Univ. Math. J..}~\textbf{29} (1980),  no.  2, 191--204. 
          
          \bibitem{Izumi}
M.  Izumi,
Inclusions of simple {$C^\ast$}-algebras,
\textit{J. Reine Angew. Math.}, \textbf{547} (2002), 97--138.


  \bibitem{Kishimoto}
  A.  Kishimoto,
  Outer automorphisms and reduced crossed products of simple
              {$C^{\ast} $}-algebras,
 \textit{Comm. Math. Phys.}~\textbf{81} (1981),  no.  3, 429--435. 
         
         \bibitem{PV}
         M.  Pimsner and D.  Voiculescu,
                    Imbedding the irrational rotation {$C^{\ast} $}-algebra into
              an {AF}-algebra,
   \textit{J. Operator Theory}~\textbf{4} (1980),  no.  2, 201--210. 
            
 \bibitem{Rieffel}
 M.  A.  Rieffel,
      Actions of finite groups on {$C^{\ast} $}-algebras,
     \textit{Math. Scand.}~\textbf{47} (1980),  no.  1, 157--176. 
             
\bibitem{Rordam}
     M.  R{\o}rdam,
     Irreducible inclusions of simple $C^*$-algebras,
     Preprint~2021. arXiv 2105.11899 
     
     \bibitem{Rosenberg}
     J.   Rosenberg,
     Appendix to: ``{C}rossed products of {UHF} algebras by product
              type actions'' 
               by {O}. {B}ratteli,
      \textit{Duke Math. J.}~\textbf{46} (1979),  no.  1, 25--26. 
             

\end{thebibliography}
\end{document}